\documentclass[a4paper,reqno]{amsart}
\usepackage{}

\usepackage{CJK}
\usepackage{mathrsfs}
\usepackage{amssymb,amsmath}
\usepackage{bbding}
\usepackage{amssymb}
\usepackage{bbm}
\usepackage{color}
\usepackage{pgf,pgfarrows,pgfnodes,pgfautomata,pgfheaps}
  
\usepackage{amstext}
\usepackage{amsmath}
\usepackage{amscd}
\usepackage{array}
\usepackage{amsthm}
\usepackage{amsfonts}
\usepackage{enumerate}
\usepackage{graphicx}
\usepackage{graphics}
\usepackage{latexsym}
\usepackage{mathrsfs}
\usepackage{mathtools}
\usepackage[all]{xy}
\usepackage[all]{xy}
\xyoption{all}

\usepackage{pstricks}
\usepackage{lscape}
\usepackage{comment}

\newtheorem{theorem}{Theorem}[section]

\newtheorem{corollary}[theorem]{Corollary}
\newtheorem{lemma}[theorem]{Lemma}
\newtheorem{proposition}[theorem]{Proposition}
\newtheorem{definition-proposition}[theorem]{Definition-Proposition}

\newtheorem{question}[theorem]{Question}

\theoremstyle{definition}
\newtheorem{definition}[theorem]{Definition}

\newtheorem{remark}[theorem]{Remark}
\newtheorem{example}[theorem]{Example}

\newcommand{\Ext}{\operatorname{Ext}\nolimits}

\newcommand{\Hom}{\operatorname{Hom}\nolimits}

\newcommand{\Tr}{\operatorname{Tr}\nolimits}
\renewcommand{\mod}{\mathsf{mod}\hspace{.01in}}

\newcommand{\add}{\mathsf{add}\hspace{.01in}}
\newcommand{\Fac}{\mathsf{Fac}\hspace{.01in}}
\newcommand{\Sub}{\mathsf{Sub}\hspace{.01in}}

\newcommand{\rad}{\operatorname{rad}\nolimits}
\newcommand{\soc}{\operatorname{soc}\nolimits}
\newcommand{\RHom}{\mathbf{R}\strut\kern-.2em\operatorname{Hom}\nolimits}

\numberwithin{equation}{section}

\usepackage{paralist}

\def\Im{\mathop{\rm Im}\nolimits}

\def\Tr{\mathop{\rm Tr}\nolimits}
\def\mod{\mathop{\rm mod}\nolimits}
\def\rad{\mathop{\rm rad}\nolimits}
\def\Hom{\mathop{\rm Hom}\nolimits}
\def\Ext{\mathop{\rm Ext}\nolimits}

\def\id{\mathop{\rm id}\nolimits}
\def\pd{\mathop{\rm pd}\nolimits}

\begin{document}
\title{A bijection theorem for Gorenstein projective $\tau$-tilting modules}
 \thanks{$*$ is the corresponding author. Z. Xie is supported by NSFC (No.12101320) and the Science and Technology Development Fund of Nanjing Medical University NMUB2020024. X. Zhang is supported by NSFC (No.12171207) and the Project Funded by the Priority
Academic Program Development of Jiangsu Higher Education Institutions and the Starting Project of Jiangsu Normal University.}

\author{Zongzhen Xie}
\address{Z. Xie: Department of Mathematics and Computer Science, School of Biomedical Engineering and Informatics, Nanjing Medical University,
Nanjing, 210029, P. R. China}
\email{xiezongzhen3@163.com}
\author{Xiaojin Zhang$^*$}
\address{X. Zhang: School of Mathematics and Statistics, Jiangsu Normal University,
Xuzhou, 221116, P. R. China}
\email{xjzhang@jsnu.edu.cn; xjzhangmaths@163.com}

\maketitle
\begin{abstract}
We introduce the notions of Gorenstein projective $\tau$-rigid modules, Gorenstein projective support $\tau$-tilting modules and  Gorenstein torsion pairs and give a Gorenstein analog to Adachi-Iyama-Reiten's bijection theorem on support $\tau$-tilting modules. More precisely, for an algebra $\Lambda$, We prove that there is a bijection between the set of Gorenstein projective support $\tau$-tilting modules and the set of functorially finite Gorenstein projective torsion classes. As an application, we introduce the notion of CM-$\tau$-tilting finite algebras and show that $\Lambda$ is CM-$\tau$-tilting finite if and only if $\Lambda^{\rm {op}}$ is CM-$\tau$-tilting finite. Moreover, we show that the Bongartz completion of a Gorenstein projective $\tau$-rigid module need not be a Gorenstein projective $\tau$-tilting module.
\end{abstract}

\section{Introduction}

 Gorenstein projective modules are essential in Gorenstein homological algebra, commutative algebra and the representation theory of algebras. They can be back to Auslander and Bridger's modules of G-dimension zero in 1969. In 1995, Enochs and Jenda \cite{EJ1,EJ2} defined the Gorenstein projective modules over an arbitrary ring as a generalization of modules of G-dimension $0$. From then on, Gorenstein homological algebra became fruitful. For the recent work on Gorenstein projective modules over finite dimensional algebras, we refer to \cite{CSZ,HuLXZ,K,RZ1,RZ2,RZ3}.

On the other hand, $\tau$-tilting theory was introduced by Adachi, Iyama and Reiten in 2014. It was shown \cite{AIR} that $\tau$-tilting theory was closely related to silting theory \cite{AiI} and cluster tilting theory \cite{IY}. In $\tau$-tilting theory, support $\tau$-tilting modules are important objects since the mutations of support $\tau$-tilting modules always exist. For the recent development on support $\tau$-tilting modules over special algebras, we refer to \cite{AiH,IZ1,IZ2,KK,S,XZZ,Zi}.

It is well-known that Gorenstein projective tilting modules are trivial modules. But there are a lot of Gorenstein projective $\tau$-tilting modules since all $\tau$-tilting modules over a self-injective algebra are Gorenstein projective ! In the present paper, we combine Gorenstein projective modules with $\tau$-tilting modules to build a new theory which is an intersection between Gorenstein homological algebra and $\tau$-tilting theory. We should remark that it was shown by Li and the second author \cite{LZ2} that Gorenstein projective support $\tau$-tilting modules do exist for algebras of self-injective dimension $n\geq0$ for any non-negative integer $n$. Recently, Gorenstein projetive $\tau$-tilting modules are used to determine gentle algebra of finite type by Liu and Zhang \cite{LiuZ}.

 Recall that a module $M\in \mod\Lambda$ is called $\tau$-rigid if $\Hom_{\Lambda}(M,\tau M)=0$, where $\tau$ is the Auslander-Reiten translation functor. Denote by $|M|$ the number of non-isomorphic indecomposable direct summands of $M$. $M$ is called a $\tau$-tilting module if it is $\tau$-rigid and $|M|=|\Lambda|$. Moreover, $M$ is support $\tau$-tilting if $M$ is a $\tau$-tilting module over $\Lambda/(e)$, where $e$ is an idempotent.
For a $\tau$-rigid pair $(M,P)$ in $\mod\Lambda$, we call $(M,P)$ Gorenstein projective if $M$ is Gorenstein projective. Dually, one can define Gorenstein injective $\tau^{-1}$-rigid pairs. The following result which is analog to the theorem \cite[Theorem 2.14]{AIR} gives a bijection between Gorenstein projective $\tau$-rigid pairs and Gorenstein injective $\tau^{-1}$-rigid pairs.
\begin{theorem}\label{1.1}(Theorem \ref{3.4})
Let $\Lambda$ be a finite dimensional algebra. There is a bijection between the following sets.
 \begin{enumerate}[\rm(1)]
\item The set of Gorenstein projective $\tau$-rigid pair in $\mod \Lambda$.
\item The set of Gorenstein projective $\tau$-rigid pair in $\mod \Lambda$$ ^{\mathrm{op}}$.
\item The set of Gorenstein injective $\tau^{-1}$-rigid pair in $\mod \Lambda$.
\item The set of Gorenstein injective $\tau^{-1}$-rigid pair in $\mod \Lambda$$ ^{\mathrm{op}}$.
\end{enumerate}
\end{theorem}

Recall from \cite{AIR} that a pair $(M,P)$ in $\mod\Lambda$ is called a support $\tau$-tilting pair if it is a $\tau$-rigid pair and $|P|+|M|=|\Lambda|$. We call a support $\tau$-tilting pair $(M,P)$ Gorenstein projective if $M$ is Gorenstein projective. In this case, we say that $M$ is a Gorenstein projective support $\tau$-tilting module. Dually, one can define Gorenstein injective support $\tau^{-1}$-tilting pairs and support $\tau^{-1}$-tilting modules. Let $(\mathcal{T},\mathcal{F})$ be a functorially finite torsion pair in $\mod\Lambda$. We call $(\mathcal{T},\mathcal{F})$ a Gorenstein torsion pair if the basic Ext-projective generator $M$ in $\mathcal{T}$ is Gorenstein projective. In this case, $\mathcal{T}$ is called a {\it Gorenstein projective} torsion class. Moreover, we call $(\mathcal{T},\mathcal{F})$ a {\it trivial} torsion pair if the basic Ext-projective generator $M$ is projective. We have the following main theorem.

\begin{theorem}\label{1.3}(Theorems \ref{3.7}, \ref{5.5}) Let $\Lambda$ be a finite dimensional algebra. Then there is a bijection between the following sets.
\begin{enumerate}[\rm(1)]
\item The set of Gorenstein projective support $\tau$-tilting modules in $\mod\Lambda$.
\item The set of functorially finite Gorenstein torsion pairs in $\mod\Lambda$.
\item The set of Gorenstein injective support $\tau^{-1}$-tilting modules in $\mod\Lambda$.
\item The set of Gorenstein projective support $\tau$-tilting modules in $\mod\Lambda^{\rm op}$.
\item The set of functorially finite Gorenstein torsion pairs in $\mod\Lambda^{\rm op}$.
\item The set of Gorenstein injective support $\tau^{-1}$-tilting modules in $\mod\Lambda^{\rm op}$.
\end{enumerate}
\end{theorem}

 Recall from \cite{DIJ} that algebras admitting finitely many isomorphism classes of indecomposable $\tau$-rigid modules are called $\tau$-tilting finite algebras and from \cite{C1} that algebras admitting finitely many isomorphism classes of indecomposable Gorenstein projective modules are called algebras of finite Cohen-Macaulay type (CM-finite for short). We are interested in algebras with finite Gorenstein projective $\tau$-tilting modules. We call an algebra $\Lambda$ CM-$\tau$-tilting finite if it has finitely many isomorphism classes of indecomposable Gorenstein projective $\tau$-rigid modules.  As an application of Theorem \ref{1.1}, we have the following theorem.
\begin{theorem}\label{1.2}(Theorem \ref{4.a}, Theorem \ref{4.4})
Let $\Lambda$ be a finite dimensional algebra.
 \begin{enumerate}[\rm(1)]
\item $\Lambda$ is CM-$\tau$-tilting finite if and only if $\Lambda$$^{\mathrm{op}}$ is CM-$\tau$-tilting finite.
\item Algebras of radical square zero are CM-$\tau$-tilting finite.
\end{enumerate}
\end{theorem}
Recall from \cite{AIR} that for a $\tau$-rigid module $M$ and $^{\bot}\tau M=\{N\in\mod\Lambda|\Hom(N,\tau M)=0\}$, the Ext-projective generator $T$ in $^{\bot}\tau M$ is called the Bongartz Completion of $M$. It is shown in \cite{AIR} that the Bongartz completion of a $\tau$-rigid module is a $\tau$-tilting module. A natural question is the following:

\begin{question}\label{1.4} Let $M\in\mod\Lambda$ be a Gorenstein projective $\tau$-rigid module. Is the Bongarz completion of $M$ a Gorenstein projective $\tau$-tilting module?
\end{question}

We give a counterexample to Question \ref{1.4} and show the following theorem.

\begin{theorem}\label{1.5} Let $\Lambda$ be an algebra with radical square zero and let $M\in\mod\Lambda$ be a Gorenstein projective $\tau$-rigid module. Then the Bongartz completion of $M$ is a Gorenstein projective $\tau$-tilting module.
\end{theorem}

As a straight result of Theorem 1.5, one gets that for an algebra with radical square zero $\Lambda$, every indecomposable Gorenstein projective $\tau$-rigid module in $\mod\Lambda$ is a direct summand of a Gorenstein projective $\tau$-tilting module.

This paper is organized as follows. In Section 2, we recall some preliminaries of Gorenstein projective modules, $\tau$-tilting modules and $\tau$-rigid pairs. In Section 3, we introduce the definitions of Gorenstein projective $\tau$-rigid pairs and show a bijection between Gorenstein projective $\tau$-rigid pairs and Gorenstein injective $\tau^{-1}$-rigid pairs. In Section 4, we introduce the definition of functorially finite Gorenstein torsion pairs and show Theorem \ref{1.3}. In Section 5, we give the definition of CM-$\tau$-tilting finite algebras and prove that algebras of radical square zero are CM-$\tau$-tilting finite. In Section 6, we show that the Bongartz completion of a Gorenstein projective $\tau$-rigid module need not be a Gorenstein projective $\tau$-tilting module. Moreover, we show that for an algebra with radical square zero $\Lambda$, every indecomposable Gorenstein projective $\tau$-rigid module in $\mod\Lambda$ is a direct summand of a Gorenstein projective $\tau$-tilting module.

Throughout this paper, $\Lambda$ is a finite dimensional algebra over a field $K$ and $\mod \Lambda$ is the category of finitely generated right $\Lambda$-modules.  Denote by $\mathbb{D}=\mathrm{Hom}_{K}(-,K)$ the ordinary dual and denote by $\tau$ the Auslander-Reiten translation functor.

\vspace{0.2cm}

{\bf Acknowledgements} The authors would like to thank Profs. Xiao-Wu Chen, Zhaoyong Huang, Osamu Iyama, Zhi-Wei Li and Pu Zhang for useful suggestions and discussions.
\section{Preliminaries}
In this section, we recall some basic results on Gorenstein projective (resp. injective) modules and $\tau$-tilting (resp. $\tau^{-1}$-tilting) modules.

 Denote by $\mathcal{P}(\Lambda)$ (resp. $\mathcal{I}(\Lambda)$) the category of finitely generated projective (resp. injective) $\Lambda$-modules. Firstly, we need the definition of Gorenstein projective modules and the definition of Gorenstein injective modules in \cite{EJ1,EJ2}.
\begin{definition}\label{2.1} Let $\Lambda$ be a finite dimensional algebra and $M\in\mod\Lambda$.
\begin{enumerate}[\rm(1)]
\item  $M$ is called {\it Gorenstein projective}, if there is an exact sequence $\cdots \rightarrow P_{-1} \rightarrow P_{0} \rightarrow P_{1} \rightarrow \cdots$ in $\mathcal{P}(\Lambda)$, which stays exact under $\Hom_{\Lambda}(-,\Lambda)=(-)^{\ast}$, such that $M\simeq\Im ( P_{-1} \rightarrow P_{0})$.
\item $M$ is called {\it Gorenstein injective}, if there is an exact sequence $\cdots \rightarrow I_{-1} \rightarrow I_{0} \rightarrow I_{1} \rightarrow \cdots$ in $\mathcal{I}(\Lambda)$, which stays exact under $\Hom_{\Lambda}(\mathbb{D}\Lambda,-)$, such that $M\simeq\Im ( I_{-1} \rightarrow I_{0})$.
\end{enumerate}
\end{definition}

 Denote by $\mathcal{GP}(\Lambda)$ (resp. $\mathcal{GI}(\Lambda)$) the category of all finitely generated Gorenstein projective (resp. Gorenstein injective) $\Lambda$-modules and by $\underline{\mathcal{GP}(\Lambda)}$ (resp. $\overline{\mathcal{GI}(\Lambda)}$) the stable category of $\mathcal{GP}(\Lambda)$ (resp. $\mathcal{GI}(\Lambda)$) which modulo projective modules (resp. injective modules).



\begin{remark}\label{2.3} Let $\Lambda$ be a finite dimensional algebra.
 \begin{enumerate}[\rm(1)]
\item It is shown in \cite{AuB, RZ1} that a module $M\in\mod\Lambda$ is {\it Gorenstein projective} if and only if $M\simeq M^{\ast\ast}$ and $\Ext_{\Lambda}^{i}(M,\Lambda)=\Ext_{\Lambda}^{i}(M^{\ast},\Lambda)=0$ hold for all $i\geq 1$ if and only if $\Ext_{\Lambda}^{i}(M,\Lambda)=0$ and $\Ext_{\Lambda}^{i}(\Tr M,\Lambda)=0$ hold for all $i\geq 1$.
\item The category $\mathcal{GP}(\Lambda)$ is closed under extensions, finite direct sums and kernel of epimorphisms.
\item The category $\mathcal{GI}(\Lambda)$ is closed under extensions, finite direct sums and cokernel of monomorphisms.
\item $\mathbb{D}: \mathcal{GP}(\Lambda)\rightarrow\mathcal{GI}(\Lambda^{\rm {op}})$ is a duality.
\end{enumerate}
\end{remark}

Next, we recall the notions of $\tau$-rigid (resp. $\tau^{-1}$-rigid) modules and $\tau$-tilting (resp. $\tau^{-1}$-tilting) modules from \cite{S} and \cite{AIR}.

\begin{definition}\label{2.4}Let $\Lambda$ be a finite dimensional algebra and $M\in\mod\Lambda$.
\begin{enumerate}[\rm(1)]
\item We call $M$ {\it $\tau$-rigid} if $\Hom_{\Lambda}(M,\tau M)=0$,
where $\tau$ is Auslander-Reiten translation. Moreover, $M$ is called a {\it $\tau$-tilting} module if $M$ is $\tau$-rigid
and $|M|=|\Lambda|$.
\item We call $M$ {\it support $\tau$-tilting} if there exists an idempotent $e$ of $\Lambda$ such that $M$ is a $\tau$-tilting $\Lambda/ (e)$-module.
\item We call $M $ {\it $\tau^{-1}$-rigid} if $\Hom_{\Lambda}(\tau^{-1}M, M)=0$.
 Moreover, $M $ is called a {\it $\tau^{-1}$-tilting} module if $M$ is $\tau^{-1}$-rigid
and $|M|=|\Lambda|$.
\item We call $M$ {\it support $\tau^{-1}$-tilting} if $M$ is a $\tau^{-1}$-tilting $\Lambda/(e)$-module for some idempotent $e$ of $\Lambda$.
\end{enumerate}
\end{definition}

The following definitions in \cite{AIR} are useful in this paper.

\begin{definition}\label{2.6}
Let $(M,P)$ and $(I,N)$ be pairs with $M, N\in \mod \Lambda$, $P\in \mathcal{P}(\Lambda)$ and $I\in \mathcal{I}(\Lambda)$.
 \begin{enumerate}[\rm(1)]
\item We call $(M,P)$ a {\it $\tau$-rigid pair} if $M$ is $\tau$-rigid and $\Hom_{\Lambda}(P,M)=0$.
\item We call $(M,P)$ a {\it support $\tau$-tilting pair} if $(M,P)$ is $\tau$-rigid and $|M|+|P|=|\Lambda|$.
\item We call $(I,N)$ a {\it $\tau^{-1}$-rigid pair} if $N$ is $\tau^{-1}$-rigid and $\Hom_{\Lambda}(N,I)=0$.
\item We call $(I,N)$ a {\it support $\tau^{-1}$-tilting pair} if $(I,N)$ is $\tau^{-1}$-rigid and $|N|+|I|=|\Lambda|$.
\end{enumerate}
\end{definition}

The following property of $\tau$-rigid pairs \cite[Proposition 2.16]{AIR} is essential in this paper.

\begin{lemma}\label{2.7} For any support $\tau$-tilting pair $(M, P)$ in $\mod\Lambda$, one gets a unique torsion pair $(\Fac M, \Sub(\tau M \oplus \nu P))$ in $\mod\Lambda$.
\end{lemma}

We also need the following results in \cite[Theorems 2.7, 2.15]{AIR}.

\begin{theorem}\label{2.8} For an algebra $\Lambda$, there is a bijection between the following sets.
\begin{enumerate}[\rm(1)]
\item The set of support $\tau$-tilting modules.
\item The set of functorially finite torsion classes.
\item The set of support $\tau^{-1}$-tilting modules.
\item The set of functorially finite torsion-free classes.
\end{enumerate}
\end{theorem}

\section{Gorenstein Projective $\tau$-tilting modules}
In this section, we give a bijection theorem between Gorenstein projective support $\tau$-tilting modules and Gorenstein injective support $\tau^{-1}$-tilting modules.

Firstly, we give the definitions of Gorenstein projective $\tau$-rigid pairs, Gorenstein projective support $\tau$-tilting pairs and Gorenstein projective $\tau$-tilting modules.

\begin{definition}\label{3.1}
Let $(M,P)$ be a pair in $\mod\Lambda$ with $P$ projective.
 \begin{enumerate}[\rm(1)]
\item We call a pair $(M,P)$ {\it Gorenstein projective $\tau$-rigid} if it is a $\tau$-rigid pair and $M$ is Gorenstein projective.
\item We call a pair $(M,P)$ {\it Gorenstein projective support $\tau$-tilting} if it is a support $\tau$-tilting pair and $M$ is Gorenstein projective.
\item We call $M\in\mod\Lambda$ {\it Gorenstein projective $\tau$-rigid} if it is both Gorenstein projective and $\tau$-rigid.
\item We call $M\in\mod\Lambda$ {\it Gorenstein projective  $\tau$-tilting} if it is both Gorenstein projective and $\tau$-tilting.
\end{enumerate}
\end{definition}

Similarly, one can define the {\it Gorenstein injective $\tau^{-1}$-rigid pairs}, {\it Gorenstein injective support $\tau^{-1}$-tilting pairs} and {\it Gorenstein injective $\tau^{-1}$-tilting} modules as follows.

\begin{definition}\label{3.a}
Let $(I,N)$ be a pair in $\mod\Lambda$ with $I$ injective.
 \begin{enumerate}[\rm(1)]
\item We call a pair $(I,N)$ {\it Gorenstein injective $\tau^{-1}$-rigid} if it is a $\tau^{-1}$-rigid pair and $N$ is Gorenstein injective.
\item We call a pair $(I,N)$ {\it Gorenstein injective support $\tau^{-1}$-tilting} if it is a support $\tau^{-1}$-tilting pair and $N$ is Gorenstein injective.
\item We call $N\in\mod\Lambda$ {\it Gorenstein injective $\tau^{-1}$-rigid} if it is both Gorenstein injective and $\tau^{-1}$-rigid.
\item We call $N\in\mod\Lambda$ {\it Gorenstein injective  $\tau^{-1}$-tilting} if it is both Gorenstein injective and $\tau^{-1}$-tilting.
\end{enumerate}
\end{definition}

To understand the definitions above we give the following observation.
\begin{example}\label{3.b}
\begin{enumerate}[\rm (1)]
\item If the global dimension of $\Lambda$ is finite, then the unique Gorenstein projective $\tau$-tilting module (resp. Gorenstein injective  $\tau^{-1}$-tilting module) in $\mod\Lambda$ is $\Lambda$ (resp. $\mathbb{D}\Lambda$).
\item If the global dimension of $\Lambda$ is finite, then every indecomposable Gorenstein projective $\tau$-rigid (resp. Gorenstein injective $\tau^{-1}$-rigid) module is projective (resp. injective).
\item If $\Lambda$ is self-injective, then the set of basic Gorenstein projective $\tau$-tilting (resp. Gorenstein injective  $\tau^{-1}$-tilting) modules coincides with the set of basic $\tau$-tilting (resp. $\tau^{-1}$-tilting) modules.
\item If $\Lambda$ is self-injective, the set of indecomposable Gorenstein projective $\tau$-rigid (resp. Gorenstein injective $\tau^{-1}$-rigid) modules coincides with the set of indecomposable $\tau$-rigid (resp. $\tau^{-1}$-rigid) modules.
\end{enumerate}
\end{example}

In what follows, we only focus on Gorenstein projective $\tau$-tilting modules since Gorenstein injective $\tau^{-1}$-tilting modules can be get dually. Now we give the following example to show that Gorenstein projective support $\tau$-tilting modules exists for classes of algebras.

\begin{example}\label{3.d} Let $\Lambda$ be an algebra given by the quiver $Q$: $\xymatrix{1\ar@<.2em>[r]^{a_{1}}&2\ar@<.2em>[l]^{a_{2}}\ar@<.2em>[r]^{b_{2}}&3}$ with the relations $a_{1}a_{2}=a_{2}a_{1}=0$.
\begin{enumerate}[\rm(1)]
\item $\Lambda\simeq\begin{smallmatrix} 1\\ &2\\
&&3\end{smallmatrix}\oplus \begin{smallmatrix} & &2\\ &1& &3\\
\end{smallmatrix}\oplus \begin{smallmatrix}3 \\ \end{smallmatrix}$ is an 1-Gorenstein algebra.
\item $\begin{smallmatrix} \\ &2\\
&&3\end{smallmatrix}\oplus \begin{smallmatrix} & &2\\ &1& &3\\
\end{smallmatrix}\oplus \begin{smallmatrix}3 \\ \end{smallmatrix}$ is a Gorenstein projective $\tau$-tilting module in $\mod\Lambda$.
\item $\begin{smallmatrix} \\ &2\\
&&3\end{smallmatrix}\oplus \begin{smallmatrix}3 \\ \end{smallmatrix}$, $\begin{smallmatrix}3 \\ \end{smallmatrix}$ and $ \begin{smallmatrix}0\\ \end{smallmatrix}$ are Gorenstein projective support $\tau$-tilting modules but not $\tau$-tilting modules in $\mod\Lambda$.
\end{enumerate}
\end{example}

 We should remark that a Gorenstein projective support $\tau$-tilting module $M\in \mod \Lambda$ need not be a Gorenstein projective $\tau$-tilting module in $\mod\Lambda/(e)$, where $(M,e\Lambda)$ is the corresponding support $\tau$-tilting pair.
 In the following, we give an example to show this.

\begin{example}\label{3.e} Let $\Lambda$ be given by the quiver $Q:$
$\xymatrix{1\ar[r]^{a_1}&2\ar[d]^{a_2}\\
 &3\ar[lu]^{a_3}}$ with the relations $a_1a_2=a_2a_3=a_3a_1=0$. Then
 \begin{enumerate}[\rm(1)]
 \item $\Lambda$ is a self-injective algebra.
 \item $T=\begin{smallmatrix} \\&&\\2\\
&&\end{smallmatrix}\oplus \begin{smallmatrix} & &\\ &2\\&&3\\
\end{smallmatrix}$ is a non-projective Gorenstein projective support $\tau$-tilting $\Lambda$-module and $(T,P(1))$ is a support $\tau$-tilting pair.
\item The algebra $\Lambda_1=\Lambda/(e_1)$ is given by the quiver $Q'$:$\xymatrix{2\ar[r]^{a_2}&3}$ and hence it is hereditary.
\item $T$ is a $\tau$-tilting module (in fact a tilting module) in $\mod\Lambda_1$, but it is not Gorenstein projective since all Gorenstein projective modules in $\mod\Lambda_1$ should be projective.

 \end{enumerate}
\end{example}

Denote by $I={\rm ann}_{\Lambda} T$ the right annihilator of $T\in\mod \Lambda$. We have the following proposition on the property of Gorenstein projective support $\tau$-tilting modules.
\begin{proposition}\label{3.c} Let $T$ be a Gorenstein projective support $\tau$-tilting module in $\mod\Lambda$. Then $T$ is Gorenstein projective in $\mod \Lambda/I$ if and only if $T\simeq \Lambda/I$.
\end{proposition}
\begin{proof}  We only show the necessity. Since tilting modules are precisely faithful support $\tau$-tilting modules \cite[Proposition 2.2(b)]{AIR}, one gets that $T$ should be a tilting module in $\mod \Lambda/I$. Then there is an exact sequence $0\rightarrow\Lambda/I\rightarrow T_0\rightarrow T_1\rightarrow0$ in $\mod\Lambda/I$ with $T_i\in\add_{\Lambda} T$. Note that $T$ is Gorenstein projective, then $\Ext_{\Lambda/I}^{1}(T_1, \Lambda/I)=0$. Therefore $\Lambda/I$ is a direct summand of $T_0$. Since $|T_0|\leq|T|=|\Lambda/I|$, one gets that $T\simeq \Lambda/I$.
\end{proof}

Recall from \cite{AIR} that there is a bijection between $\tau$-rigid modules in $\mod\Lambda$ and $\tau$-rigid modules in $\mod\Lambda ^{\mathrm{op}}$. We decompose $M$ as $M=M_{\mathrm{p}}\bigoplus M_{\mathrm{np}}$, where $M_{\mathrm{p}}$ is a maximal projective direct summand of $M$ and $M_{\mathrm{np}}$ has no projective direct summand of $M$. For a $\tau$-rigid pair $(M,P)$ in $\mod\Lambda$, let $(M,P)^{\dag}=(\Tr M_{\mathrm{np}}\bigoplus P^{\ast},M_{\mathrm{p}}^{\ast})=(\Tr M\bigoplus P^{\ast},M_{\mathrm{p}}^{\ast})$. For a $\tau$-rigid $\Lambda$-module $M$, we simply write $M^{\dag}=\Tr M_{\mathrm{np}}\bigoplus P^{\ast}$.

The following bijections due to Adachi, Iyama and Reiten are quite essential in this paper.

\begin{lemma}\label{3.2}
Let $\Lambda$ be an algebra. Then the functor $(-)^{\dag}$ gives the following bijections.
 \begin{enumerate}[\rm (1)]
\item $\tau$-rigid modules (pairs) in $\mod \Lambda$ $\leftrightarrow$ $\tau$-rigid modules (pairs) in $\mod\Lambda^{\mathrm{op}}$.
\item support $\tau$-tiling modules (pairs) in $\mod \Lambda$ $\leftrightarrow$ support $\tau$-tiling modules (pairs) in $\mod \Lambda^{\rm op}$.
\end{enumerate}
such that $(-)^{\dag \dag}=\id$.
\end{lemma}

Before giving the main theorem in this section, we introduce the following well-known lemma on Gorenstein projective modules. For convenience, we give a proof.
\begin{lemma}\label{3.3}
Let $\Lambda$ be an algebra.
 \begin{enumerate}[\rm(1)]
\item $(-)^*=\Hom_{\Lambda}(-,\Lambda): \mathcal{GP}(\Lambda) \rightarrow \mathcal{GP}(\Lambda^{\mathrm{op}})$ is a duality.

\item $\Omega^{1} : \underline{\mathcal{GP}(\Lambda)} \rightarrow \underline{\mathcal{GP}(\Lambda)}$ is an equivalence of categories.

\item $\Tr : \underline{\mathcal{GP}(\Lambda)} \rightarrow \underline{\mathcal{GP}(\Lambda^{\mathrm{op}})}$ is a duality.
\end{enumerate}

\end{lemma}

\begin{proof}
(1) Let $M\in \mathcal{GP}(\Lambda)$ and $\mathbf{P}: \cdots \rightarrow P_{-1} \rightarrow P_{0} \rightarrow P_{1} \rightarrow \cdots$ be a complete projective resolution of $M$. Then $\Hom_{\Lambda}(\mathbf{P},\Lambda)$ is an exact sequence of right projective $\Lambda$-modules, and $\Hom_{\Lambda}(\Hom_{\Lambda}(\mathbf{P},$$ _{\Lambda}\Lambda),\Lambda_{\Lambda}) \simeq \Hom_{\Lambda}(_{\Lambda}\Lambda,\mathbf{P})$. Thus $\Hom_{\Lambda}(\Hom_{\Lambda}(\mathbf{P},$$_{\Lambda}\Lambda),\Lambda_{\Lambda})$ is exact. By definition $\Hom_{\Lambda}(\mathbf{P},$$_{\Lambda}\Lambda)$ is a complete projective resolution of $\Hom_{\Lambda}(M,$$_{\Lambda}\Lambda)$ and $M \simeq \Hom_{\Lambda}(\Hom_{\Lambda}(M,$$_{\Lambda}\Lambda),\Lambda_{\Lambda})$ by Five Lemma. Thus $\Hom_{\Lambda}(M,$$_{\Lambda}\Lambda) \in \mathcal{GP}(\Lambda^{\mathrm{op}})$. Similarly, one can prove the other side.

(2) For $M\in \mathcal{GP}(\Lambda)$, $\Ext_{\Lambda}^{1}(M,\Lambda)=0$, we have that $\Omega^{1}: \Hom (\underline{M},\underline{N})\rightarrow \Hom (\underline{\Omega^{1} M},\underline{\Omega^{1} N})$ is an isomorphism for all $N \in \mod \Lambda$, so that $\Omega^{1} : \underline{\mathcal{GP}(\Lambda)} \rightarrow \underline{\mathcal{GP}(\Lambda)}$ is fully faithful. $\Omega^{1}$ is dense since a Gorenstein projective module without nonzero projective summands is an arbitrary syzygy.

(3) Let $M\in \underline{\mathcal{GP}(\Lambda)}$ admit no projective direct summands. Consider a minimal projective resolution $\cdots \rightarrow P_{1} \rightarrow P_{0} \rightarrow M \rightarrow 0$, we have the exact sequence $0 \rightarrow M^{*} \rightarrow P_{0}^{*} \rightarrow P_{1}^{*}\rightarrow \cdots$ and $P_{0}^{*} \rightarrow P_{1}^{*}\rightarrow \Tr M \rightarrow 0$, showing that $\Tr M \in \underline{\mathcal{GP}(\Lambda^{\mathrm{op}})}$ by Remark \ref{2.3}. One can prove the other side similarly.
\end{proof}
Now we can give the following bijection theorem on Gorenstein projective $\tau$-rigid modules.
\begin{theorem}\label{3.4}
Let $\Lambda$ be an algebra. There is a bijection between the following sets.
 \begin{enumerate}[\rm(1)]
\item The set of Gorenstein projective $\tau$-rigid pair in $\mod \Lambda$.
\item The set of Gorenstein projective $\tau$-rigid pair in $\mod \Lambda$$ ^{\mathrm{op}}$.
\item The set of Gorenstein injective $\tau^{-1}$-rigid pair in $\mod \Lambda$.
\item The set of Gorenstein injective $\tau^{-1}$-rigid pair in $\mod \Lambda$$ ^{\mathrm{op}}$.
\end{enumerate}
\end{theorem}

\begin{proof}
We only prove the bijection between $(1)$ and $(2)$ since the bijection between $(2)$ and $(3)$ can be showed by the duality $\mathbb{D}: \mathcal{GP}(\Lambda)\rightarrow \mathcal{GI}(\Lambda^{\mathrm{op}})$ and $\mathbb{D}$: $\tau$-rigid  modules in $\mod\Lambda \rightarrow$ $\tau^{-1}$-rigid modules in $\mod\Lambda^{\rm op}$. And the proof of the bijection between $(3)$ and $(4)$ is similar to $(1)$ and $(2)$.

Let $(M,P)$ be a $\tau$-rigid pair in $\mod \Lambda$.
Then $(M,P)$ is a $\tau$-rigid pair in $\mod \Lambda$ if and only if $(M,P)^{\dag}$ is a $\tau$-rigid pair in $\mod \Lambda$$ ^{\mathrm{op}}$.
$M$ is Gorenstein projective module provide that $M \simeq M^{\ast\ast}$ and $\Ext_{\Lambda}^{i}(M,\Lambda)=\Ext_{\Lambda}^{i}(M^{\ast},\Lambda)=0$.
For a Gorenstein projective module $M$, there is a $\Hom_{\Lambda}(-,\Lambda)$ exact exact sequence
$$\cdots \rightarrow P_{-2} \rightarrow P_{-1} \rightarrow P_{0} \rightarrow P_{1} \rightarrow P_{2} \rightarrow \cdots $$
of projective modules such that $M=\Im (P_{-1}\rightarrow P_{0})$.
Applying $(-)^{\ast}$ to above exact sequence, we have
$$0 \rightarrow M^{\ast} \rightarrow P_{-1}^{\ast} \rightarrow P_{-2}^{\ast} \rightarrow (\Omega^{1} M)^{\ast} \rightarrow 0$$
By Lemma \ref{3.3} we have $\Omega^{1} M$ is Gorenstein projective since $M$ is Gorenstein projective. Thus $\Tr M \simeq (\Omega^{1} M)^{\ast}$ is a Gorenstein projective $\Lambda$$ ^{\mathrm{op}}$-module and $\Tr M \bigoplus P^{\ast}=M^{\dagger}$ is a Gorenstein projective $\Lambda$$ ^{\mathrm{op}}$-module.

Conversely, if $(M,P)$ is a $\tau$-rigid pair in $\mod \Lambda^{\mathrm{op}}$, the proof is similar.

Thus, we have a bijection between Gorenstein projective $\tau$-rigid pairs in $\mod \Lambda$ and Gorenstein projective $\tau^{-1}$-rigid pairs in $\mod\Lambda^{\rm op}$.
\end{proof}

Applying to Gorenstein projective support $\tau$-tilting pairs, we have

\begin{theorem}\label{3.7}Let $\Lambda$ be an algebra. There is a bijection between the following sets.
 \begin{enumerate}[\rm(1)]
\item The set of Gorenstein projective support $\tau$-tilting pair in $\mod \Lambda$.
\item The set of Gorenstein projective support $\tau$-tilting pair in $\mod \Lambda$$ ^{\mathrm{op}}$.
\item The set of Gorenstein injective support $\tau^{-1}$-tilting pair in $\mod \Lambda$.
\item The set of Gorenstein injective support $\tau^{-1}$-tilting pair in $\mod \Lambda$$ ^{\mathrm{op}}$.
\end{enumerate}
\end{theorem}

\begin{proof} $(1)\Leftrightarrow(2)$ Let $(M,P)$ be a Gorenstein projective support $\tau$-tilting pair in $\mod\Lambda$.  By Theorem \ref{3.4},
then $(M,P)^{\dag}$ is a Gorenstein projective $\tau$-rigid pair in $\mod\Lambda^{\rm op}$. Then by Lemma \ref{3.2}(2), $(M,P)^{\dag}$ is a support $\tau$-tilting pair. One gets the assertion easily since $(-)^{\dag}$ is a duality.

$(1)\Leftrightarrow(3)$  Let $(M,P)$ be a Gorenstein projective support $\tau$-tilting pair in $\mod\Lambda$. One gets a bijective map via $(M,P)\rightarrow \mathbb{D}(M,P)^{\dag}$ since both $(-)^{\dag}$ and $\mathbb{D}$ are dualities.

$(2)\Leftrightarrow(4)$ is similar to $(1)\Leftrightarrow(3)$.
We are done.
\end{proof}
Then we have the following corollary on Gorenstein projective $\tau$-tilting modules.
\begin{corollary}\label{3.5}
Let $\Lambda$ be an algebra and let $M\in \mod\Lambda$ have no projective direct summand. Then the following are equivalent.
\begin{enumerate}[\rm (1)]
\item $M$ is a Gorenstein projective $\tau$-tilting $\Lambda$-module.
\item $M^{\dag}(=\Tr M)$ is a Gorenstein projective $\tau$-tilting $\Lambda^{\mathrm{op}}$-module.
 \item $\tau M$ is a Gorenstein injective $\tau^{-1}$-tilting $\Lambda$-module.
\end{enumerate}
\end{corollary}
\begin{proof}
$M=M_{\mathrm{np}}$ and $(M,0)$ is a $\tau$-rigid pair since $M$ has no projective direct summand. So $M^{\dag}=\Tr M$. $M$ is a $\tau$-rigid $\Lambda$-module if and only if $\Tr M$ is a $\tau$-rigid $\Lambda^{\mathrm{op}}$-module. By \cite[Proposition 2.2]{AIR} sincere support $\tau$-tilting modules are $\tau$-tilting, we can get that $M$ is a $\tau$-tilting $\Lambda$-module if and only if $\Tr M$ is a $\tau$-tilting $\Lambda^{\mathrm{op}}$-module. Then we can get the assertion by Theorem \ref{3.7}.
\end{proof}
Next, we give an example to illustrate the Gorenstein projective $\tau$-tilting modules without projective direct summands do exist! Applying the functor $\tau$, one can get the Gorenstein injective $\tau^{-1}$-tilting modules without injective direct summands in $\mod\Lambda$.
\begin{example}\label{3.6}
Let $\Lambda$ be an algebra given by the quiver $Q$: $\xymatrix{1\ar@<.2em>[r]^{a_{1}}&2\ar@<.2em>[l]^{b_{2}}\ar@<.2em>[r]^{a_{2}}&3\ar@<.2em>[l]^{b_{1}}}$ with the relations $a_{1}b_{2}=a_{2}b_{1}=0$ and $b_{1}a_{2}=b_{2}a_{1}$.
The  support $\tau$-tilting quiver of $\Lambda$ is
$$
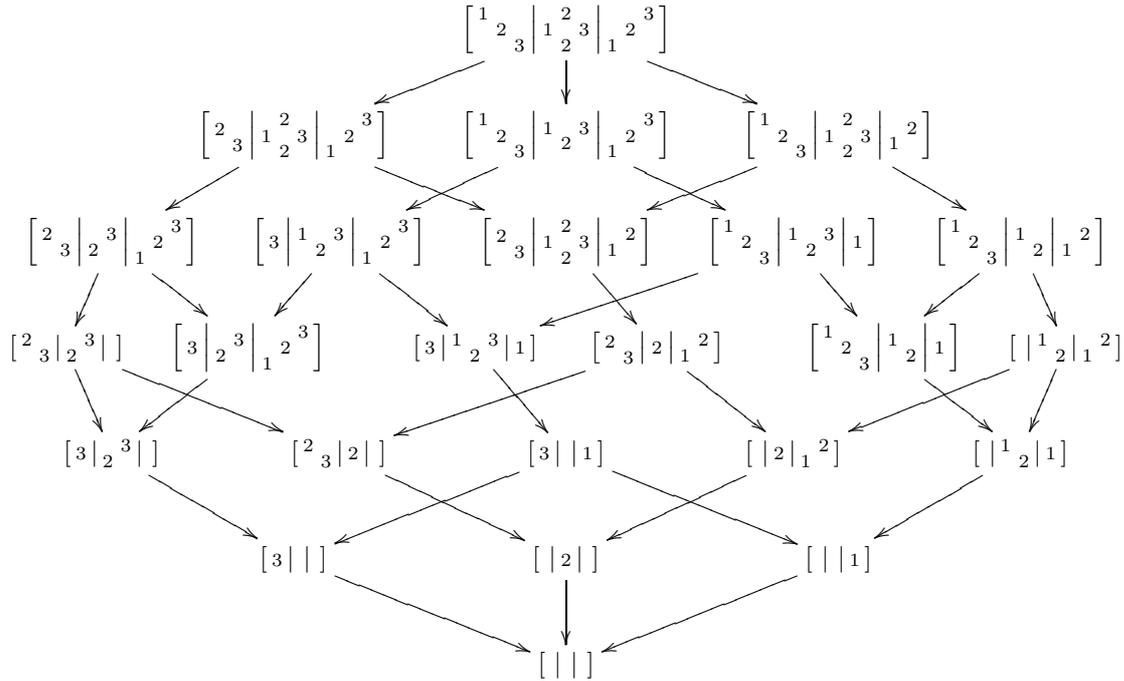
\begin{xy}
0;<3.4pt,0pt>:<0pt,2.5pt>::
(0,0)
*+{ \left[\begin{smallmatrix} 1\\ &2\\
&&3\end{smallmatrix}\middle| \begin{smallmatrix} &2\\ 1&&3\\
&2\\ \end{smallmatrix}\middle| \begin{smallmatrix} &&3\\ &2\\
1\\ \end{smallmatrix}\right]}="1", (-30,-16)
*+{\left[\begin{smallmatrix} \\ 2\\
&3\end{smallmatrix}\middle| \begin{smallmatrix} &2\\ 1&&3\\
&2\\ \end{smallmatrix}\middle| \begin{smallmatrix} &&3\\ &2\\
1\\ \end{smallmatrix}\right]}="2", (0,-16)
*+{\left[\begin{smallmatrix} 1\\ &2\\
&&3\end{smallmatrix}\middle| \begin{smallmatrix} 1&&3\\
&2\\ \end{smallmatrix}\middle| \begin{smallmatrix} &&3\\ &2\\
1\\ \end{smallmatrix}\right]}="3", (30,-16)
*+{\left[\begin{smallmatrix} 1\\ &2\\
&&3\end{smallmatrix}\middle| \begin{smallmatrix} &2\\ 1&&3\\
&2\\ \end{smallmatrix}\middle| \begin{smallmatrix} &2\\
1\end{smallmatrix}\right]}="4", (-50,-32)
*+{\left[\begin{smallmatrix} 2\\ &3\end{smallmatrix}\middle|
\begin{smallmatrix} &3\\ 2\\ \end{smallmatrix}\middle|
\begin{smallmatrix} &&3\\ &2\\1\\ \end{smallmatrix}\right]}="5", (-25,-32)
*+{\left[\begin{smallmatrix} 3\end{smallmatrix}\middle|
\begin{smallmatrix} 1&&3\\ &2\\ \end{smallmatrix}\middle|
\begin{smallmatrix} &&3\\ &2\\
1\\ \end{smallmatrix}\right]}="6", (0,-32)
*+{\left[\begin{smallmatrix} \\ 2\\
&3\end{smallmatrix}\middle| \begin{smallmatrix} &2\\ 1&&3\\
&2\\ \end{smallmatrix}\middle| \begin{smallmatrix} &2\\
1\end{smallmatrix}\right]}="7", (25,-32)
*+{\left[\begin{smallmatrix} 1\\ &2\\
&&3\end{smallmatrix}\middle| \begin{smallmatrix} 1&&3\\
&2\\ \end{smallmatrix}\middle| \begin{smallmatrix}
1\end{smallmatrix}\right]}="8", (50,-32)
*+{\left[\begin{smallmatrix} 1\\ &2\\
&&3\end{smallmatrix}\middle| \begin{smallmatrix} 1\\
&2\end{smallmatrix}\middle| \begin{smallmatrix} &2\\
1\end{smallmatrix}\right]}="9", (-55,-48)
*+{\left[\begin{smallmatrix} 2\\ &3\end{smallmatrix}\middle|
\begin{smallmatrix} &3\\ 2\\ \end{smallmatrix}\middle|\ \right]}="10",
(-35,-48)
*+{\left[\begin{smallmatrix}
3\end{smallmatrix}\middle| \begin{smallmatrix} &3\\
2\\ \end{smallmatrix}\middle| \begin{smallmatrix} &&3\\
&2\\1\\ \end{smallmatrix}\right]}="11", (-10,-48)
*+{\left[\begin{smallmatrix}
3\end{smallmatrix}\middle| \begin{smallmatrix} 1&&3\\
&2\\ \end{smallmatrix}\middle| \begin{smallmatrix}
1\end{smallmatrix}\right]}="12", (10,-48)
*+{\left[\begin{smallmatrix} \\ 2\\
&3\end{smallmatrix}\middle| \begin{smallmatrix}
2\end{smallmatrix}\middle| \begin{smallmatrix} &2\\
1\end{smallmatrix}\right]}="13", (35,-48)
*+{\left[\begin{smallmatrix} 1\\ &2\\
&&3\end{smallmatrix}\middle| \begin{smallmatrix} 1\\
&2\end{smallmatrix}\middle| \begin{smallmatrix}
1\end{smallmatrix}\right]}="14", (55,-48)
*+{\left[\ \middle|\begin{smallmatrix} 1\\ &2\end{smallmatrix}\middle|
\begin{smallmatrix} &2\\ 1\end{smallmatrix}\right]}="15",
(-50,-64)
*+{\left[\begin{smallmatrix}
3\end{smallmatrix}\middle| \begin{smallmatrix} &3\\
2\\ \end{smallmatrix}\middle|\ \right]}="16", (-25,-64)
*+{\left[\begin{smallmatrix} 2\\
&3\end{smallmatrix}\middle| \begin{smallmatrix}
2\end{smallmatrix}\middle|\ \right]}="17", (-0,-64)
*+{\left[\begin{smallmatrix} 3\end{smallmatrix}\middle|\ \middle|
\begin{smallmatrix} 1\end{smallmatrix}\right]}="18", (25,-64)
*+{\left[\ \middle|\begin{smallmatrix}
2\end{smallmatrix}\middle| \begin{smallmatrix} &2\\
1\end{smallmatrix}\right]}="19", (50,-64)
*+{\left[\ \middle|\begin{smallmatrix} 1\\
&2\end{smallmatrix}\middle| \begin{smallmatrix}
1\end{smallmatrix}\right]}="20", (-30,-80)
*+{\left[\begin{smallmatrix}
3\end{smallmatrix}\middle|\ \middle|\ \right]}="21", (0,-80)
*+{\left[\ \middle|\begin{smallmatrix}
2\end{smallmatrix}\middle|\ \right]}="22", (30,-80)
*+{\left[\ \middle|\ \middle|\begin{smallmatrix}
1\end{smallmatrix}\right]}="23", (0,-96)
*+{\left[\ \middle|\ \middle|\ \right]}="24",  \ar"1";"2",
\ar"1";"3", \ar"1";"4", \ar"2";"5", \ar"2";"7", \ar"3";"6",
\ar"3";"8", \ar"4";"7", \ar"4";"9", \ar"5";"10", \ar"5";"11",
\ar"6";"11", \ar"6";"12", \ar"7";"13", \ar"8";"12", \ar"8";"14",
\ar"9";"14", \ar"9";"15", \ar"10";"16", \ar"10";"17", \ar"11";"16",
\ar"12";"18", \ar"13";"17", \ar"13";"19", \ar"14";"20",
\ar"15";"19", \ar"15";"20", \ar"16";"21", \ar"17";"22",
\ar"18";"21", \ar"18";"23", \ar"19";"22", \ar"20";"23",
\ar"21";"24", \ar"22";"24", \ar"23";"24", \end{xy}$$

The middle two modules in the fourth line are Gorenstein projective $\tau$-tilting $\Lambda$-modules without projective direct summands.
 \end{example}

\section{Gorenstein torsion pairs}

In this section, we give the definition of Gorenstein torsion pairs and then show the bijections between Gorenstein projective support $\tau$-tilting modules and Gorenstein torsion pairs.

To give the main results in this section, we give the following definition of Gorenstein torsion pairs.
\begin{definition}\label{5.1} Let $\Lambda$ be an algebra and let $(\mathcal{T},\mathcal{F})$ be a functorially finite torsion pair in $\mod\Lambda$. We call $(\mathcal{T},\mathcal{F})$ Gorenstein if the basic Ext-projective generator $M$ in $\mathcal{T}$ is Gorenstein projective. In this case, $\mathcal{T}$ is called a {\it Gorenstein projective} torsion class. Moreover, we call $(\mathcal{T},\mathcal{F})$ {\it trivial} if the basic Ext-projective generator $M$ is projective.
\end{definition}

Now we give an example to show that there are a lot of Gorenstein torsion pairs.

\begin{example}\label{5.2}Let $\Lambda$ be an algebra.
\begin{enumerate}[\rm (1)]
 \item Both $(\mod\Lambda,0)$ and $(0,\mod\Lambda)$ are functorially finite Gorenstein torsion pairs in $\mod\Lambda$.
 \item If $\Lambda$ is self-injective and $T$ is a $\tau$-tilting module in $\mod\Lambda$, then $(\Fac T, \Sub\tau T)$ is a functorially finite Gorenstein torsion pair.
 \end{enumerate}
\end{example}
We also give the following property of functorially finite Gorenstein torsion pairs.
\begin{proposition}\label{5.3} Let $\Lambda$ be an algebra and let $(\mathcal{T},\mathcal{F})$ be a functorially finite torsion pair in $\mod\Lambda$.
\begin{enumerate}[\rm(1)]
\item $(\mathcal{T},\mathcal{F})$ is Gorenstein if and only if the basic Ext-injective cogenerator in $\mathcal{F}$ is Gorenstein injective.
\item $(\mathcal{T},\mathcal{F})$ is trivial if and only if every Ext-injective object in $\mathcal{F}$ is injective.
\item $(\mathcal{T},\mathcal{F})$ is Gorenstein if and only if $(\mathbb{D}\mathcal{F},\mathbb{D}\mathcal{T})$ is Gorenstein in $\mod\Lambda^{\rm op}$.
\end{enumerate}
\end{proposition}
\begin{proof} (1) $\Rightarrow$ Since $(\mathcal{T},\mathcal{F})$ is functorially finite, then there is an Ext-projective generator $T\in\mathcal{T}$ such that $\mathcal{T}=\Fac T$. By Theorem \ref{2.8}, $T$ is a support $\tau$-tilting module. Let $(T,P)$ be the support $\tau$-tilting pair given by $T$. By Lemma \ref{2.7}, one gets a functorially finite torsion pair $(\Fac T, \Sub(\tau T\oplus \nu P))$ which coincides with $(\mathcal{T},\mathcal{F})$. Since $(\mathcal{T},\mathcal{F})$ is functorially finite Gorenstein torsion pair, then $T$ is Gorenstein projective. By Lemma \ref{3.3}, $\tau T\oplus \nu P$ is Gorenstein injective which is an Ext-injective cogenerator in $\mathcal{F}$.

$\Leftarrow$ Since $(\mathcal{T},\mathcal{F})$ is functorially finite, then there is an Ext-injective module $M\in\mathcal{F}$ such that $\mathcal{F}=\Sub M$. By Theorem \ref{2.8}, $M$ is a support $\tau^{-1}$-tilting module. Similarly, one gets a support $\tau$-tilting pair $(T,P)$ in $\mod \Lambda$ such that $M\simeq \tau T\oplus \nu P$. Since $M$ is Gorenstein injective, then $T$ is Gorenstein projective by Lemma \ref{3.3}.

(2) One can get the assertion by a similar proof to (1).

(3) This is a straight result of (1).
\end{proof}

As a corollary, we can give the following characterizations of algebras.

\begin{corollary}\label{5.4} Let $\Lambda$ be an algebra.
\begin{enumerate}[\rm (1)]
\item If $\Lambda$ is of finite global dimension, then all functorially finite Gorenstein torsion pairs are trivial.
\item If $\Lambda$ is self-injective, then every functorially finite torsion pair is Gorenstein.
\end{enumerate}
\end{corollary}
\begin{proof}(1) For any functorially finite torsion pair $(\mathcal{T},\mathcal{F})$, then $\mathcal{T}\simeq \Fac M$ with $M$ Ext-projective in $\mathcal{T}$. If $(\mathcal{T},\mathcal{F})$ is Gorenstein, then $M$ is Gorenstein projective. Since $\Lambda$ is of finite projective dimension, then $M$ is projective.

(2) Since $\Lambda$ is self-injective, one gets that every module in $\mod\Lambda$ is Gorenstein projective. The assertion holds.
\end{proof}

We should remark that the converse of Corollary \ref{5.4}(1) is not true in general since all functorially finite Gorenstein torsion pairs over a CM-free algebra are always trivial. Moreover, the converse of Corollary \ref{5.4}(2) fails in general since the unique functorially finite torsion pair over a local algebra is trivial.

Now we are in a position to show the following main result.

\begin{theorem}\label{5.5} Let $\Lambda$ be an algebra. Then there is a bijection between the following sets.
\begin{enumerate}[\rm(1)]
\item The set of Gorenstein projective support $\tau$-tilting modules.
\item The set of functorially finite Gorenstein torsion pairs.
\item The set of Gorenstein injective support $\tau^{-1}$-tilting modules.
\end{enumerate}
\end{theorem}
\begin{proof} By Theorem \ref{2.8}, one gets a bijection between the set of support $\tau$-tilting modules and the set of functorially finite torsion pairs via $T\rightarrow \Fac T$. By the definition of Gorenstein projective support $\tau$-tilting modules and the definition of functorially finite Gorenstein torsion pairs, the map also works in our case.
\end{proof}

As a consequence, we have the following corollary.

\begin{corollary}\label{5.6} Let $\Lambda$ be an 1-Iwanaga-Gorenstein algebra, that is, $\id\Lambda_{\Lambda}=\id{_{\Lambda}\Lambda}\leq1$. The following are equivalent.
\begin{enumerate}[\rm(1)]
\item $\Lambda$ is self-injective.
\item All functorially finite torsion pairs are Gorenstein.
\item All support $\tau$-tilting modules are Gorenstein projective.
\item All support $\tau^{-1}$-tilting modules are Gorenstein injective.
\item All $\tau$-rigid modules are Gorenstein projective.
\item All $\tau^{-1}$-rigid modules are Gorenstein injective.
\end{enumerate}
\end{corollary}

\begin{proof} $(1)\Rightarrow(2)$ is obvious and $(2)\Rightarrow(3)$ follows from Theorem \ref{5.5}.

We show $(3)\Rightarrow(1)$.
By \cite[Theorem 1.1]{XZZ}, $\mathbb{D}\Lambda$ is a $\tau$-tilting module (in fact a tilting module) in $\mod\Lambda$. Therefore $\mathbb{D}\Lambda$ is Gorenstein projective. Since $\Lambda$ is $1$-Iwanaga-Gorenstein, then $\pd \mathbb{D}\Lambda\leq 1$. By Remark \ref{2.3}, one gets that $\mathbb{D}\Lambda$ is projective, that is, $\Lambda$ is self-injective.

$(2)\Leftrightarrow(4)$ follows from Theorem \ref{5.5}. $(3)\Leftrightarrow(5)$ since every $\tau$-rigid module is a direct summand of a support $\tau$-tilting module. $(4)\Leftrightarrow(6)$ follows from the fact every $\tau^{-1}$-rigid module is a direct summand of a support $\tau^{-1}$-tilting module.
\end{proof}

\section{CM-$\tau$-tilting finite algebras}
In this section, we give the definition of CM-$\tau$-tilting finite algebras, show the basic properties and prove that algebras of radical square zero are CM-$\tau$-tilting finite.

Recall that a finite dimensional algebra $\Lambda$ is called {\it CM-finite} if it admits finite number of indecomposable Gorenstein projective modules. For more details on CM-finite algebras we refer to \cite{C1, LZ1}.
And $\Lambda$ is called {\it$\tau$-tilting finite} if there are only finitely many isomorphism classes of indecomposable $\tau$-rigid $\Lambda$-modules \cite{DIJ}. Now we introduce the definition of CM-$\tau$-tilting finite algebras.

\begin{definition}\label{4.1}
We call an algebra {\it CM-$\tau$-tilting finite} if it only has finitely many isomorphism classes of indecomposable Gorenstein projective $\tau$-rigid $\Lambda$-modules.
\end{definition}

\begin{remark}\label{4.2}
Both CM-finite algebras and $\tau$-tilting finite algebras are CM-$\tau$-tilting finite algebras.
\end{remark}

In the following we give an example to show that the definition of CM-$\tau$-tilting finite algebra is far from meaningless.

\begin{example}\label{4.5}
 \begin{enumerate}[\rm(1)]
\item Let $\Lambda$ be an algebra given by the quiver $Q$: $1\rightrightarrows2$. Then $\Lambda$ is $\tau$-tilting infinite but CM-$\tau$-tilting finite.
\item The preprojective algebras of Dynkin type $A_{n}$ with $n\geq 5$, is CM-infinite but CM-$\tau$-tilting finite.
\end{enumerate}
\end{example}

Now we show the basic properties of CM-$\tau$-tilting finite algebras. Following \cite{DIJ} we have the following theorem.

\begin{theorem}\label{4.a} Let $\Lambda$ be an algebra. Then the following are equivalent.
\begin{enumerate}[\rm (1)]
\item $\Lambda$ is CM-$\tau$-tilting finite.
\item There are finitely many indecomposable Gorenstein projective $\tau$-rigid modules in $\mod\Lambda$.
\item $\Lambda^{\rm op}$ is CM-$\tau$-tilting finite.
\item There are finitely many indecomposable Gorenstein projective $\tau$-rigid modules in $\mod\Lambda^{\rm op}$.
\end{enumerate}
\end{theorem}

\begin{proof} Clearly, $(1)\Leftrightarrow (2)$ and $(3)\Leftrightarrow (4)$ follow from the definition. Now it suffices to show $(2)\Leftrightarrow (4)$. $(2)\Leftrightarrow (4)$ follows from Theorem \ref{3.4}.
\end{proof}

Although algebras of radical square zero are not necessary to be $\tau$-tilting finite, we show that they are CM-$\tau$-tilting finite.

\begin{theorem}\label{4.4}
Algebras of radical square zero are indeed CM-finite, and hence CM-$\tau$-tilting finite.
\end{theorem}

\begin{proof}
Let $\Lambda$ be an algebra with radical square zero. Then by \cite{C2} $\Lambda$ is either self-injective or CM-free, that is, all Gorenstein projective modules in $\mod\Lambda$ are projective. We only need to show that $\Lambda$ is a Nakayama algebra if it is self-injective.

For any indecomposable projective $P \in \mod \Lambda$, there is a simple module $S$ such that $P_{0}(S)\simeq P$. Then we get the following exact sequence $0\rightarrow \Omega^{1}(S) \rightarrow P_{0}(S) \rightarrow S \rightarrow 0$. Since $\Lambda$ is of radical square zero, then $\Omega^{1}(S) \simeq \rad P_{0}(S)$ is semi-simple. On the other hand, the fact $\Lambda$ is self-injective implies that $P_{0}(S)$ is indecomposable injective. Then $\soc P_{0}(S)$ is simple, and hence $\Omega^{1}(S) \subset \soc P_{0}(S)$ is simple. So $P$ admits a unique composition series.
\end{proof}

Recall that a quotient algebra of a $\tau$-tilting finite algebra is again $\tau$-tilting finite. It is natural to
consider the quotient algebras of CM-$\tau$-tilting algebras. However, we have the following example in \cite{IZ2} to show that a quotient algebra of CM-finite algebra need not be CM-finite.

\begin{example}\label{4.6} Let $\Lambda$ be the Auslander algebra of $K[x]/(x^n)$ with $n\geq 6$. Then
\begin{enumerate}[\rm (1)]
\item The Auslander algebra $\Lambda$ is presented by the quiver
\[\xymatrix{
1\ar@<2pt>[r]^{a_1}&2\ar@<2pt>[r]^{a_2}\ar@<2pt>[l]^{b_2}&3\ar@<2pt>[r]^{a_3}\ar@<2pt>[l]^{b_3}&\cdots\ar@<2pt>[r]^{a_{n-2}}\ar@<2pt>[l]^{b_4}&n-1\ar@<2pt>[r]^{a_{n-1}}\ar@<2pt>[l]^{b_{n-1}}&n\ar@<2pt>[l]^{b_n}
}\] with relations $a_{1} b_{2}= 0$ and $a_{i} b_{i+1} =b_{i}
a_{i-1}$ for any $2 \leq i \leq n-1$.
\item Let $\Gamma=\Lambda/(e_n)$, then $\Gamma$ is a preprojective algebra of $A_{n-1}$ of infinite representation type.
\item $\Lambda$ is CM-finite but $\Gamma$ is not CM-finite.
\end{enumerate}
\end{example}

We end this paper with the following question.
\begin{question}\label{4.7} Let $\Lambda$ be a CM-$\tau$-tilting finite algebra and $\Gamma$ be a quotient algebra of $\Lambda$. Is $\Gamma$ CM-$\tau$-tilting finite?
\end{question}

We remark that it is shown in \cite[Theorem 3.13]{LZ2} that $T_n(A)$ is $CM$-$\tau$-tilting finite implies that $A$ is $CM$-$\tau$-tilting finite. In the following we can give the following theorem which generalizes \cite[Theorem 3.13]{LZ2}.

\begin{theorem} Let $\Lambda$ and $\Gamma$ be two algebras. If $\Lambda\otimes_K \Gamma$ is a $CM$-$\tau$-tilting finite, then both $\Lambda$ and $\Gamma$ are $CM$-$\tau$-tilting finite.
\end{theorem}

\begin{proof} We only show the case for $\Lambda$. On the contrary, suppose that $\Lambda$ admits infinite number of non-isomorphic Gorenstein projective $\tau$-rigid modules $M_1, M_2,\cdots, M_n,\cdots$. Let $P\in\mod\Gamma$ be an indecomposable projective module. Then by \cite[Proposition 3.5]{LZ2}, one gets that $M_i\otimes_K P$ is a $\tau$-rigid module in $\mod(\Lambda\otimes_K\Gamma)$ for $n\geq 1$. By using \cite[Proposition 2.6]{HuLXZ}, $M_i\otimes_K P$ is a Gorenstein projective module in $\mod(\Lambda\otimes_K\Gamma)$ for $n\geq 1$. Finally, one gets that $M_i\otimes_K P$ is indecomposable by \cite[Proposition 3.1]{LZ2}. Then one gets infinite number of non-isomorphic indecomposable Gorenstein projective $\tau$-rigid modules: $M_1\otimes_K P, M_2\otimes_K P, M_3\otimes_K P, \cdots$. This is a contradiction since $\Lambda\otimes_K \Gamma$ is $CM$-$\tau$-tilting finite.
\end{proof}

\section{The Bongartz completion of a Gorenstein projective $\tau$-rigid module}

In this section, we give a counterexample to Question 1.5 and show that a Gorenstein projective $\tau$-rigid module need not be a direct summand of a Gorenstein projective $\tau$-tilting module.

\begin{example}\label{6.1} Let $\Lambda$ be an algebra give by the quiver $Q:$
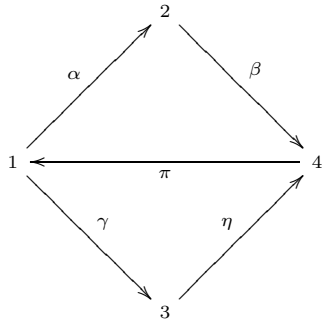
$$\begin{xy}
 (0,0)*+{\begin{smallmatrix}1 \end{smallmatrix}}="1",
(20,20)*+{\begin{smallmatrix}2\end{smallmatrix}}="2",
(20,-20)*+{\begin{smallmatrix}3\end{smallmatrix}}="3",
(40,0)*+{\begin{smallmatrix}4\end{smallmatrix}}="4",

\ar^{\alpha}"1";"2", \ar^{\beta}"2";"4", \ar^{\gamma}"1";"3", \ar^{\eta}"3";"4", \ar^{\pi}"4";"1",
\end{xy}$$
with the relations $\beta\alpha=\eta\gamma$, $\pi\beta=\alpha\pi=0$ and $\pi\eta=\gamma\pi=0$. Then

(1) $\Lambda$ is an $1$-Gorenstein algebra, that is, $\id_{\Lambda} \Lambda=1$.

(2) $S(1)=1$ is a Gorenstein projective $\tau$-rigid module and $^{\bot}\tau S(1)=\add$
$\begin{smallmatrix}1\\ 3 \end{smallmatrix}\oplus
\begin{smallmatrix}1\\2\end{smallmatrix}\oplus
\begin{smallmatrix}4\\1\end{smallmatrix}\oplus
\begin{smallmatrix}4\end{smallmatrix}\oplus
\begin{smallmatrix}1\end{smallmatrix}$.

(3) $\begin{smallmatrix}1\\ 3 \end{smallmatrix}$ and $\begin{smallmatrix}1\\2\end{smallmatrix}$ are indecomposable injective modules and hence of finite projective dimension. But they are not Gorenstein projective. Otherwise, $\begin{smallmatrix}1\\ 3 \end{smallmatrix}$ and $\begin{smallmatrix}1\\2\end{smallmatrix}$ are projective, a contradiction.

(4) The Bongartz completion of $S(1)$ is not a Gorenstein projective $\tau$-tilting module.
\end{example}

In the following we show that there are algebras such that the Bongartz completions of every Gorenstein projective $\tau$-rigid modules are Gorenstein projective $\tau$-tilting modules.

\begin{lemma}\label{6.2} Let $\Lambda$ be an algebra. If $\Lambda$ is either self-injective or $CM$-free, then every Gorenstein projective $\tau$-rigid module is a direct summand of a Gorenstein projective $\tau$-tilting module.
\end{lemma}

\begin{proof} If $\Lambda$ is self-injective, then every module in $\mod\Lambda$ is Gorenstein projective. The assertion holds. If $\Lambda$ is $CM$-free, then every Gorenstein projective $\tau$-rigid module is projective. The assertion holds.
\end{proof}

Now we are in the position to show the main result in this section.

\begin{theorem}\label{6.3} Let $\Lambda$ be an algebra with radical square zero. Then the Bongartz completion of a Gorenstein projective $\tau$-rigid module is a Gorenstein projective $\tau$-tilting module. And hence every Gorenstein projective $\tau$-rigid module is a direct summand of a Gorenstein projective $\tau$-tilting module.
\end{theorem}

\begin{proof} By \cite{C2}, $\Lambda$ is either self-injective or $CM$-free. Then by Lemma \ref{6.2}, the assertion holds.
\end{proof}

We should remark that there does exist an algebra which is neither self-injective nor $CM$-free such that every Gorenstein projective $\tau$-rigid module is a direct summand of a Gorenstein projective $\tau$-tilting module.

\noindent{\bf Competing interests:} The authors declare none.


\end{document}